\documentclass{amsart}
\usepackage{amssymb}
\usepackage{hyperref}
\hypersetup{%
	colorlinks,
	breaklinks=true,
	plainpages=false,%
	citecolor=blue,
	linkcolor=blue,
	urlcolor=blue,
	bookmarksopen=true,%
	bookmarksnumbered=false,%
	bookmarksdepth=5%
}

\newtheorem{theorem}{Theorem}[section]
\newtheorem{lemma}[theorem]{Lemma}
\newtheorem{proposition}[theorem]{Proposition}

\newtheorem{corollary}[theorem]{Corollary}

\theoremstyle{definition}
\newtheorem{definition}[theorem]{Definition}

\theoremstyle{remark}
\newtheorem{remark}[theorem]{Remark}

\newcommand{\pu}{\mathrm{PU}}

\newcommand{\C}{\mathbb{C}}
\newcommand{\reals}{\mathbb{R}}
\newcommand{\Heis}{\mathcal{H}}
\newcommand{\F}{\mathbb{F}}

\newcommand{\Proy}{\mathbb P_{\mathbb C}}
\newcommand{\Hip}{\mathbb{H}}
\newcommand{\Sp}{\mathbb{S}}

\newcommand{\vb}{\mathbf}

\begin{document}
	\title{Smooth knot limit sets of 
	the complex hyperbolic plane}
	
	\author{Waldemar Barrera}
	\address{ Facultad de Matem\'aticas, Universidad
		Aut\'onoma de Yucat\'an, Anillo Perif\'erico Norte Tablaje Cat
		13615, M\'erida, Yucat\'an, M\'exico. }
	\email{bvargas@correo.uady.mx}
	%\thanks{Research supported by CONACYT Project Number 176680}

	\author{Ren\'e Garc\'ia}
	\address { Facultad de Matem\'aticas, Universidad
		Aut\'onoma de Yucat\'an, Anillo Perif\'erico Norte Tablaje Cat
		13615, M\'erida, Yucat\'an, M\'exico. }
	\email{rene.garcia@correo.uady.mx}
	
	\author{Juan Pablo Navarrete}
	\address{ Facultad de Matem\'aticas, Universidad
		Aut\'onoma de Yucat\'an, Anillo Perif\'erico Norte Tablaje Cat
		13615, M\'erida, Yucat\'an, M\'exico. }
	\email{jp.navarrete@correo.uady.mx}
	
	

	\begin{abstract}
		It is shown that if a regular knot of class $C^2$ is embedded in the boundary of the complex hyperbolic plane as the limit set of a discrete subgroup of $\mathrm{PU}(2,1)$ then it is either a chain or an $\reals$-circle.
	\end{abstract}

	\maketitle 
	
    \section{Introduction}
	
	One of the most beautiful objects in mathematics is the three dimensional unitary sphere 
	$$\Sp^3 = \left\{(x_1, x_2, x_3, x_4) \in \mathbb{R}^4 \mid x_1^2+x_2^2+x_3^2+x_4^2= 1\right\}.$$

This set is in fact the ideal boundary of the real hyperbolic space of dimension 4. The  isometry group of this space is the  group of conformal transformations of $\Sp^3$, denoted by  $\textrm{Conf}(\Sp^3)$, which is identified with the Lorentz group $\textrm{O}(3,1)$ ( see \cite{Ka}, \cite{Ra}). Also, the space $\Sp^3$ can be considered as the boundary of complex hyperbolic space $\mathbb{H}^2_{\mathbb{C}}\subset \mathbb{C}^2$. 
The isometry group of $\mathbb{H}^2_{\mathbb{C}}$ is the group generated by
$\textrm{PU}(2,1)$ and the antiholomorphic involution $\rho_0(Z)=\bar{Z}$, where
$Z= (z_1, z_2)\in \mathbb{C}^2$ (see \cite{Go}).	 Both groups,
$\textrm{Conf}(\Sp^3)$ and
$\textrm{Isom}(\mathbb{H}^2_{\mathbb{C}})$ act by diffeomorphisms on $\Sp^3$, although   the action of  $\textrm{Isom}(\mathbb{H}^2_{\mathbb{C}})$ is not by conformal automorphisms,  but by contact morphisms, that means that the  elements of $\textrm{Isom}(\mathbb{H}^2_{\mathbb{C}})$ preserve the standard contact structure of $\Sp^3$~\cite{Go}. The discrete subgroups of 
$\textrm{Conf}(\Sp^3)$  and $\textrm{Isom}(\mathbb{H}^2_{\mathbb{C}})$ are considered higher Kleinian groups, for which there are many properties similar to classical Kleinian groups, i.e. discrete subgroups of $\textrm{PSL}(2,\C)$. Nevertheless there are several differences, which make them interesting objects to study, we recommend the following references about the subject: \cite{Ahl}, \cite{Go}, \cite{Ka}, \cite{Tu}.	

If $G$ is a discrete subgroup of  $\textrm{Conf}(\Sp^3)$  or $\textrm{PU}(2,1)$,  the limit set  $L(G)$ is defined as the set of   cluster points of the orbit of any point in $\Sp^3$. 
It is a well known fact that if $L(G)$ is a tame knot, then 
it is trivial~\cite{Ku}. 
In the case $G \subset \textrm{Conf}(\Sp^3)$, 
the result is stronger 
since 
%If the limit set $L(G)$ is a smooth knot, 
%whenever $G$ is a discrete subgroup of $\textrm{Conf}(\Sp^3)$~\cite{Ku},
$L(G)$ is  a round circle \cite{Ka}.
%, whereas for 
%$G \subset \textrm{PU}(2,1)$ a characterization of $L(G)$ is still open. 
%on the other hand, 
%when $G \subset \textrm{PU}(2,1)$ it is 
%only known that $L(G)$ is trivial. 
%In this article we show that if $L(G)$ is a regular smooth knot, then $L(G)$ is a chain or an $\mathbb{R}$-circle.  We summarize this result in the following theorem.
In this article we prove the following theorem.
%Throughout this paper, $G \subset \textrm{PU}(2,1)$ is a discrete subgroup
%acting on complex hyperbolic space $\Hip ^2 _{\C}$, and the limit set 
%$L(G) \subset \partial \Hip ^2 _{\C}$ is the image of a regular knot of class $C^2$ parametrized by
%$\gamma : \Sp ^1 \to L(G) \subset \partial\Hip ^2 _{\C}$.
\begin{theorem}\label{main}
Let $G \subset \textrm{PU}(2,1)$ be a discrete subgroup acting 
on the complex hyperbolic space $\Hip ^2 _{\C}$. 
If the limit set $L(G) \subset \partial \Hip ^2 _{\C}$ is the image of a $C^2$ regular knot, $\gamma:\Sp^1 \to \partial\Hip^2_\C$,  then it is the boundary 
of a totally geodesic subspace of $\Hip^2_\C$, in other words it is either a 
chain or an $\reals$-circle.
\end{theorem}

It is a well-known fact that the groups $\textrm{PU}(1,1)$ and $\textrm{PO}(2,1)$ 
can be embedded in $\textrm{PU}(2,1)$ preserving the vertical chain and the canonical $\mathbb{R}$-circle, respectively, 
see \cite{Go}. Also, any chain and any $\mathbb{R}$-circle can be obtained as the limit set of a conjugate subgroup of these embedded groups. Moreover, we have the following straightforward consequence of the Theorem \ref{main}:

\begin{corollary}\label{cor} If $G$ is a discrete subgroup of $\textrm{PU}(2,1)$ and the limit set $L(G)$ is the image of a 
class $C^2$ regular knot then $G$ is conjugate to a discrete subgroup
of $\textrm{PU}(1,1)$ or $\textrm{PO}(2,1)$.
\end{corollary}

%%%%%%%%%%%%%%%%%%%%%%%%%%%%%%%%%%%%%%%%%%%%%%%%%%%%%%%%%%%%%%%%%%
In~\cite{yue}, Yue proves a more general result for cocompact 
lattices of $\mathrm{SO}(n,1)$ and $\mathrm{SU}(n,1)$, 
$n \geq 2$. 
Roughly speaking, he proves that for any cocompact lattice $G_0$ of $\textrm{SO}(n,1)$ or $\mathrm{SU}(n,1)$ and any injective 
representation of $G_0$ whose image is denoted by $G$, the Hausdorff dimension of the limit 
set $L(G)$ is bounded below by the Hausdorff dimension of the limit 
set $L(G_0)$ with equality if and only if $G$ stabilizes 
a totally geodesic copy of $\Hip^n_\reals$ or $\Hip^n_\C$ respectively. The Theorems~1.3 and~1.4 of Yue imply the Theorem~\ref{main} 
and the Corollary~\ref{cor} with the additional hypothesis that  the 
group $G$ acts 
convex cocompactly. However, this additional hypothesis is not necessary 
because the condition that the limit set $L(G)$ is of class $C^2$ 
implies that the Hausdorff dimension of the limit set is 1  and allows to use techniques of 
classical differential geometry to prove the 
Theorem~\ref{main}.
%%%%%%%%%%%%%%%%%%%%%%%%%%%%%%%%%%%

This article is organized as follows: In Section~\ref{sec:preliminares} we outline the required preliminaries about complex hyperbolic geometry. In Section~\ref{sec:real-circles} the notion of $\reals$-circle is introduced together with basic notions of contact geometry. In Section~\ref{sec:chains} the notions of chain and tangent chain to a smooth curve in a point are introduced. In Proposition~\ref{invariancebypu21} we show that tangent chains to curves are mapped by elements of $\mathrm{PU}(2,1)$ to tangent chains to the corresponding mapped curves. In Proposition~\ref{prop:legendrianknot} it is shown that a knot is Legendrian at a point if and only if the tangent chain at the point is degenerate. In Section~\ref{sec:proof} we prove the Theorem~\ref{main}, the idea is to prove, using the dynamics of a loxodromic element, that if the limit set of a discrete subgroup of $\mathrm{PU}(2,1)$ is a knot then it is either a chain or a Legendrian knot.  Finally, if the limit set is a Legendrian knot,   
with help of a suitable Heisenberg coordinate system, we take the vertical projection to the horizontal plane and prove that the projection is an infinite $\reals$-circle. 

\section{Preliminaries}\label{sec:preliminares}
\subsection{Projective geometry}
We denote by $\mathbb{F}$ the field $\mathbb{C}$ or $\mathbb{R}$,
the projective space $\mathbb{P}^2_{\mathbb{F}}$ is defined
as
$$\mathbb{P}^2_\Bbb{F}:=(\mathbb{F}^{3}\setminus \{\mathbf{0}\})/\mathbb{F}^*,$$
 where $\mathbb{F}^* =\mathbb{F}\setminus \{0\}$ acts on 
 $\mathbb{F}^3\setminus\{\mathbf{0}\}$ by the usual scalar multiplication. 
The space  $\mathbb{P}^2_{\mathbb{F}}$ is called the complex  
projective plane when $\mathbb{F}=\C$, and it is called the real projective plane 
when  $ \mathbb{F}=\mathbb{R}$.
Let $[\mbox{ }]:\mathbb{F}^{3}\setminus\{\mathbf{0}\}\rightarrow
\mathbb{P}^{2}_{\mathbb{F}}$ be   the quotient map, if
$\mathbf{v}=(x,y,z)\in \mathbb{F}^3\setminus\{\mathbf{0}\}$ then we write
$[\mathbf{v}]=[x:y:z]$.  
A set of the form
$$\ell= \{[x:y:z] \in \mathbb{P}^2 _{\mathbb{F}} : A x+ By+Cz=0\}$$ for some $A,B,C \in
\mathbb{F}$ not all zero is called an $\mathbb{F}$-line (or line when there is no confusion).
We notice that any two distinct points $p, q \in \mathbb{P} ^2 _{\mathbb{F}}$ define a
unique $\mathbb{F}$-line containing $p$ and  $q$. 
Consider the action of $\F ^*= \F \setminus\{0\}$  on
$\textrm{GL}(3,\F)$ given by the usual scalar
multiplication, then
$$\textrm{PGL}(3,\F)=\textrm{GL}(3,\F )/\F ^* $$ is a Lie group
whose elements are called projective transformations.  Let $[\mbox{
}]:\textrm{GL}(3,\mathbb{F})\rightarrow \textrm{PGL}(3,\F)$
be   the quotient map. If  $g\in \textrm{PGL}(3,\F)$ and
$\mathbf{g} \in \textrm{GL}(3,\F)$, we say that $\mathbf{g}$
is a lift of $g$ whenever $[\mathbf{g}]=g$. One can show that
$\textrm{PGL}(3,\F)$ is a Lie group  that acts
transitively, effectively and by diffeomorphisms on
$\mathbb{P}^2_{\F}$ by
$[\mathbf{g}]([\mathbf{v}])=[\mathbf{g}\mathbf{v}]$, where
$\mathbf{v}\in \mathbb{F}^3\setminus\{\mathbf{0}\}$ and $\mathbf{g}\in
\textrm{GL}(3, \F)$~\cite{Go}. Notice that any element $g \in
\textrm{PGL}(3, \F)$ maps $\F$-lines to $\F$-lines. 
The real projective plane $\mathbb{P}^2 _{\mathbb{R}}$ can be embedded in a natural way in
the complex projective plane $\Proy ^2$, in the following way:
$$\mathbb{P}^2 _{\mathbb{R}}  \hookrightarrow \Proy ^2$$
$$[x:y:z] \mapsto [x:y:z].$$
In what follows, $\mathbb{P}^2 _{\mathbb{R}}$ denotes the image of this embedding.

\subsection{Complex hyperbolic geometry}\label{chg}

Let $\C ^{2,1}$ denote $\C ^3$ equipped with the Hermitian form
$$\langle \mathbf{z}, \mathbf{w}\rangle _1=
z_1 \overline{w}_1+z_2\overline{w}_2-z_3\overline{w}_3,$$ where
$\mathbf{z}=(z_1,z_2,z_3)$, $\mathbf{w}=(w_1,w_2,w_3)$. 
Denote by
$$V_{-}=\{\mathbf{z} \in \C ^{2,1} : \langle \mathbf{z}, \mathbf{z} \rangle _1<0\},$$
$$V_0=\{\mathbf{z} \in \C ^{2,1} \setminus \{\mathbf{0}\} : \langle\mathbf{z}, \mathbf{z}\rangle _1=0\},$$
$$V_{+}=\{\mathbf{z} \in \C ^{2,1} : \langle\mathbf{z}, \mathbf{z} \rangle _1>0\},$$
the sets of negative, null and positive vectors in $\C ^{2,1}
\setminus \{\mathbf{0}\}$, respectively.
The projectivization of the set of negative vectors,
\begin{eqnarray*}
[V_{-}] =  \{ [z_1:z_2:1] \in \Proy ^2 : |z_1|^2+|z_2|^2<1 \},
\end{eqnarray*}
 is a complex $2$-dimensional open ball in $\Proy^2$. Moreover, $[V_{-}]$ equipped
 with the quadratic form induced by the Hermitian form $\langle \cdot, \cdot\rangle _1$ is a
 model for the complex hyperbolic space $\Hip ^2 _{\C}$.
 The projectivization of the set of null vectors,
\begin{eqnarray*}
[V_{0}] = \{ [z_1:z_2:1] \in \Proy ^2 : |z_1|^2+|z_2|^2 = 1 \},
\end{eqnarray*}
is a $3$-sphere in $\Proy^2$ and it is the boundary of $\Hip ^2
_{\C}$, denoted $\partial \Hip ^2 _{\C}$. 
Finally, the projectivization of the set of positive vectors,
\begin{eqnarray*}
[V_{+}] & = & \{[z_1:z_2:z_3] \in \Proy^2 :
|z_1|^2+|z_2|^2-|z_3|^2>0\},
\end{eqnarray*}
is the complement in $\Proy^2$ of the complex $2$-dimensional closed
ball $\overline{\Hip^2 _{\C}}=\Hip^2 _{\C} \cup \partial \Hip^2
_{\C}$.

The group of holomorphic isometries of $\Hip ^2 _{\C}$ is
$\textrm{PU}(2,1)$, the projectivization in $\textrm{PGL}(3,\C)$ of
the unitary group, $\textrm{U}(2,1)$, respect to the Hermitian form $\langle \cdot , \cdot \rangle _1$:
$$
\textrm{U}(2,1) = \{ \mathbf{g} \in \textrm{GL}(3, \Bbb{C}): \langle \mathbf{g}
\mathbf{z}, \mathbf{g} \mathbf{w}\rangle _1=\langle\mathbf{z}, \mathbf{w}\rangle _1 \}.
$$
The group $\textrm{PU}(2,1)$ acts transitively in $\Hip ^2 _{\C}$
and by diffeomorphisms in the boundary $\partial \Hip ^2 _{\C} \cong
\Sp^3$. 
Another fact we
use along this paper is the following: Given any point
$p=[w_1:w_2:w_3] \in
\partial \Hip ^2 _{\C}$, there exists a unique complex line,
denoted $\ell _p$, tangent to $\partial \Hip ^2 _{\C}$ at $p$.
Moreover, $\ell _p$ is  the set:
$$\{[z_1: z_2: z_3] \in \Proy^2 : z_1 \overline{w}_1 +
z_2 \overline{w}_2-z_3 \overline{w}_3=0\}.$$

If we consider $\Bbb{C}^3$ with the Hermitian form
$$\langle \mathbf{z},\mathbf{w}\rangle _2=z_1 \overline{w}_3+z_2 \overline{w}_2+z_3 \overline{w}_1,$$
where $\mathbf{z}=(z_1, z_2, z_3)$ and $\mathbf{w}=(w_1, w_2, w_3)$,
then we have that
$$\langle C \mathbf{z}, C \mathbf{w}\rangle _1=\langle \mathbf{z}, \mathbf{w} \rangle _2,$$
where $C$ is the Cayley matrix:
\begin{displaymath}
\frac{1}{\sqrt{2}}\left(
\begin{array}{ccc}
1 & 0 & 1 \\
0 & \sqrt{2} & 0 \\
1 & 0 & -1
\end{array}
\right).
\end{displaymath}
Hence $C(V_{-}), C(V_0), C(V_+)$ are the sets of negative, null and
positive vectors for $\langle \cdot , \cdot \rangle _2$, respectively. The projectivization of
$C(V_{-})$ equipped with the Hermitian form $\langle \cdot , \cdot \rangle _2$ is the Siegel
model for complex hyperbolic space
$$ \{[z_1: z_2 : 1] \in \Proy ^2 : 2 \Re(z_1)+|z_2|^2<0\} $$
and its boundary is the set
$$\{[z_1:z_2:1] \in \Proy^2 : 2 \Re(z_1)+|z_2|^2=0\} \cup\{[1:0:0]\}.$$
Any finite point in this boundary, can be written in the form
$$[-|\zeta|^2+i v:\sqrt{2} \zeta: 1], $$
for some $(\zeta,v) \in \Bbb{C}\times \Bbb{R}$. Hence there is a
natural identification of this boundary set with the one point
compactification of the Heisenberg space $\mathcal{H}=\C \times
\Bbb{R}$. For more details on
complex hyperbolic geometry, see the book \cite{Go}.

\subsection{The limit set}

\begin{definition}If $G$ is a discrete subgroup of $\textrm{PU}(2,1)$,
the \emph{limit set} of $G$, denoted by $L(G)$, is defined as the set of cluster points of
the $G$-orbit of any point in ${\Hip ^2 _{\C}}$. 
\end{definition}

Some useful properties of the limit set are the following:
\begin{enumerate}
\item The limit $L(G)$ does not depend on the choice of the point in ${\Hip ^2 _{\C}}$.

\item The limit set $L(G)$ is a closed $G$-invariant set and it is minimal on the family of closed $G$-invariant sets
with more than two points. 

\item If $L(G)$ has more than two points and 
 $x \in L(G)$ is any point, 
then the $G$-orbit of $x$ is dense in $L(G)$.

\item If $L(G)$ contains at least two points then there exists a loxodromic element
in $G$. A loxodromic element is any element in $\textrm{PU}(2,1)$ 
with precisely two fixed points in $\partial \Hip ^2 _{\C}$, 
one %of these fixed points 
is \emph{attracting} and the other 
is \emph{repelling}.
\end{enumerate}

For more details see \cite{ChG,  kamiya, Ka}.
\subsection{The Hermitian Cross Product}
If $\mathbf{z}, \mathbf{w} \in \C ^{2,1}$ then the Hermitian cross product of
$\mathbf{z}=(z_1, z_2, z_3)$ and $\mathbf{w}=(w_1, w_2, w_3)$, %denoted $\mathbf{z} \boxtimes \mathbf{w}$,
with respect to the Hermitian form
$\langle \cdot , \cdot \rangle _1 $, 
is defined as 
\begin{displaymath}
\mathbf{z} \boxtimes _1 \mathbf{w} =\left| 
\begin{array}{ccc}
\mathbf{i} & \mathbf{j} & \mathbf{k} \\
\overline{z_1} & \overline{z_2} & -\overline{z_3} \\
\overline{w_1} & \overline{w_2} & -\overline{w_3}
\end{array}
\right|.
\end{displaymath} 
Analogously, the Hermitian cross product, respect to the Hermitian form
$\langle \cdot , \cdot \rangle _2 $ is defined as
\begin{displaymath}
\mathbf{z} \boxtimes_2 \mathbf{w} =\left| 
\begin{array}{ccc}
\mathbf{i} & \mathbf{j} & \mathbf{k} \\
\overline{z_3} & \overline{z_2} & \overline{z_1} \\
\overline{w_3} & \overline{w_2} & \overline{w_1}
\end{array}
\right|.
\end{displaymath} 
When there is no danger of confusion, we will omit the subscripts on the notation of these Hermitian cross products. 
These products share several similarities to the real cross product in $\reals^3$, in particular:
\begin{enumerate}
    \item $\mathbf{z}\boxtimes \mathbf{w} \neq \mathbf{0}$ if and only if $\mathbf{z}$ and $\mathbf{w}$ are linearly independent.
    \item $\mathbf{z}\boxtimes\mathbf{w}$ is orthogonal with respect to the corresponding Hermitian form to both $\mathbf{z}$ and $\mathbf{w}$.
    \item It is skew linear and  anti-commutative.
    \item For any $\mathbf{g} \in \mathrm{SU}(2,1)$, $\mathbf{g}\,(\mathbf{z}\boxtimes \mathbf{w})=\mathbf{g}(\mathbf{z})\boxtimes \mathbf{g}(\mathbf{w})$. 
\end{enumerate}

\begin{definition}
If $z=[\mathbf{z}]$ and $w=[\mathbf{w}]$ are two distinct points in $\Proy ^2$ then 
$\mathbf{z} \boxtimes \mathbf{w} \ne \mathbf{0}$ and
$$z\boxtimes w =[\mathbf{z} \boxtimes \mathbf{w}]$$
is a well-defined point in $\Proy ^2$.
\end{definition}

\section{$\reals$-circles}\label{sec:real-circles}
The boundaries of totally geodesic subspaces of $\Hip^2_\C$ which are  isometric to the Beltrami model of real hyperbolic geometry are called  $\reals$-circles.
\begin{definition}
The canonical $\reals$-circle $R_0$ is defined as $\mathbb{P}^2_\reals \cap \partial\Hip^2_\C$, any other $\reals$-circle is the translate of $R_0$ by an element of $\textrm{PU}(2,1)$.
\end{definition}

In Heisenberg coordinates, the $\reals$-circles are infinite or finite according to whether they pass through the point at infinity or not. All the infinite $\reals$-circles passing through the origin are horizontal lines $(te^{i\theta},0)$, $t \in \reals$ and $\theta \in [0,2\pi]$, any other infinite $\reals$-circle can be obtained by Heisenberg translations, thus, the infinite $\reals$-circles are simply straight lines. 
Likewise, finite $\reals$-circles are obtained by complex dilations, rotations and Heisenberg translations of the finite $\reals$-circle 
\begin{align*}
    \left\{\left(i\sqrt{\cos(2\theta)}\,e^{i\theta}, -\sin(2\theta)\right) \mid \theta \in [-\pi/4, \pi/4) \cup (3\pi/4, 5\pi/4] \right\}.
\end{align*}
More details can be found in~\cite{parker}.

\subsection{ Contact geometry}
Let $E \subset TM$ be a hyperplane field of real codimension 1 on a
smooth manifold $M$,
if there is a 1-form  $\omega : TM \to \reals$ such that for any $x \in M$, 
$E_x$ is equal to $\ker(\omega_x)$, then we call $\omega$ a contact form 
provided  $d\omega_x$ is nondegenerate for all $x \in M$.  
Moreover, the hyperplane field $E$ is called a contact structure. A curve such
that each tangent vector is in $E$ is called \emph{Legendrian}, we will focus
on the three dimensional Heisenberg space with
coordinates $(\zeta,v) \in \mathcal{H}$ and the group
structure determined by the product
\begin{align*}
 (\zeta_1,v_1)\cdot(\zeta_2,v_2) = (\zeta_1 + \zeta_2, v_1 +
\eta(\zeta_1, \zeta_2) + v_2),
\end{align*}
where  $\eta(\zeta_1,\zeta_2) = \textrm{Im}(\overline{\zeta_1}\,\zeta_2)$ is the
symplectic form determined by the complex structure of $\C$. In Heisenberg
space the form $\omega = dv - \eta(\zeta,d\zeta)$ induces the contact structure,
hence if $\gamma: \Sp^1 \to \mathcal{H}$ is a Legendrian knot with
parametrization $\gamma = (\zeta, v)$, the condition that
$\dot \gamma$ is everywhere in $\ker\omega$ is equivalent to the equation
\begin{align*}
  \dot v - \eta(\zeta, \dot\zeta) = 0,
\end{align*}
showing that Legendrian knots are determined by their vertical projections $\gamma \mapsto \zeta$. In particular, if $v(s)$ is constant, this implies that $\gamma(s)$ is a straight line.

 The $\reals$-circles are Legendrian knots: First, 
 the canonical $\reals$-circle in Heisenberg coordinates is $(t,0)$, $t \in \reals$ which is 
 a Legendrian subspace. Since $\textrm{PU}(2,1)$ preserves the contact structure, 
 any other $\reals$-circle  
 is everywhere tangent to the contact structure as well. 

 \section{Chains}\label{sec:chains}
Complex geodesics 
are the totally geodesic subspaces obtained as the intersection of a complex line and $\Hip
^2 _{\C}$. The boundary at infinity of a complex
 geodesic is a circle obtained as the intersection of $\partial \Hip ^2
 _{\C}$ and a complex line, these circles are called \emph{chains}. 

If $p=[\mathbf{v}] \in \Proy^2 \setminus \overline{\Hip ^2 _{\C}}$ then
$\mathbf{v}$ is a positive vector, so the orthogonal complement 
$\langle \mathbf{v} \rangle ^{\perp}$ respect to the first Hermitian form, is a two dimensional subspace of
$\Bbb{C}^{2,1}$ and it induces a complex line, $\ell _p$, called the
\emph{polar line to} $p$. The chain, $\mathcal{C}_p$, obtained as the
intersection $\ell _p \cap \partial \Hip ^2 _{\C}$ is \emph{the
polar chain to} $p$.
The points in $\partial \Hip
^2 _{\C}$ are considered as chains and we call them \emph{degenerate
chains}. 
Conversely, if $\ell$ is a complex line transversal to $\partial
\Hip ^2 _{\C}$, then we can write $\ell=[L\setminus\{\mathbf{0}\}]$
where $L$ is a two dimensional complex vector subspace of $\C
^{2,1}$. Moreover, the orthogonal complement of $L$, respect to the
Hermitian form $\langle \cdot, \cdot \rangle_1$, is a one dimensional
complex subspace of $\C ^{2,1}$ which induces a point in $\Proy ^2
\setminus \overline{\Hip ^2 _{\C}}$, this point is called the
\emph{polar point to} the line $\ell$.

There is a natural identification of a chain $\partial \Hip ^2 _{\C}
\cap \ell$ with the polar point to $\ell$. In fact, there is a
bijection,  
between the space of chains and the complement of the
complex hyperbolic space $\Proy ^2 \setminus \Hip ^2 _{\C}$. We remark that for a degenerate chain
$\{p\}=\partial \Hip ^2 _{\C} \cap \ell$ the corresponding point is
$p \in
\partial \Hip ^2 _{\C}$.

\subsection{Tangent chains}
Let $\gamma : \Sp ^1 \to \partial \Hip ^2 _{\C}$ be a $C^1$ regular knot which we identify with 
the curve
$$\gamma: [0,1] \to \partial \Hip ^2 _{\C},$$
$$t \mapsto \gamma(e^{2\pi i t}).$$
If $t_0,t \in [0,1)$ are distinct then
$\gamma(t_0) \ne \gamma(t)$
and the polar chain to the point
$\gamma(t_0) \boxtimes \gamma(t)$ is the unique chain passing through the points
$\gamma(t_0)$ and $\gamma(t)$.  
We notice that as $t \to t _0$, the chain passing through $\gamma (t _0)$ and $\gamma(t)$
goes to the tangent chain to $\gamma$ at $\gamma(t _0)$. However, this tangent chain
cannot be defined as the polar chain to 
$\gamma(t _0) \boxtimes \gamma(t _0)$ because the Hermitian cross product of a
vector with itself is equal to the zero vector.

In order to define the tangent chain to $\gamma$ at $t_0$, we use the following notation:
$\mathbf{v}(t) \in \C ^{2,1}\setminus \{\mathbf{0}\}$ is a vector satisfying
$\gamma(t)=[\mathbf{v}(t)]$. Moreover, we can assume that the curve
$$\mathbf{v}: [0,1] \to \C ^{2,1}\setminus \{ \mathbf{0}\}$$
is of class $C^1$.
Now, the chain passing through the points
$\gamma(t_0)=[\mathbf{v}(t_0)]$ and $\gamma(t)=[\mathbf{v}(t)]$ is the polar chain to the point
$$[\mathbf{v}(t_0) \boxtimes \mathbf{v}(t)]=
\left[\mathbf{v}(t_0) \boxtimes \left(\frac{\mathbf{v}(t)-\mathbf{v}(t_0)}{t-t_0}\right)\right].$$
Since $$\lim _{t \to t_0} \mathbf{v}(t_0) \boxtimes \left(\frac{\mathbf{v}(t)-\mathbf{v}(t_0)}{t-t_0}\right)
= \mathbf{v}(t_0) \boxtimes \mathbf{v}'(t_0),$$
we have the following definition (see \cite{Miner}):
\begin{definition}
We define the \emph{tangent chain} to $\gamma$ at $\gamma(t_0)$ as the polar chain to the point
$$\gamma(t_0)\boxtimes \gamma '(t_0)=[\mathbf{v}(t_0) \boxtimes \mathbf{v}'(t_0)] \in 
\Proy ^2 \setminus \Hip ^2 _{\C}.$$
The tangent chain to $\gamma$ at $t_0$ is denoted by $\mathcal{T}_{\gamma}(t_0)$.
\end{definition}

\begin{remark}\label{taniswelldef}
The tangent chain $\mathcal{T}_{\gamma}(t_0)$ does not depend of the lifted curve $\mathbf{v}:[0,1]\to \C ^{2,1}$ or
parametrization.
\end{remark}
In fact, if $\mathbf{w}(t)=\lambda(t)\mathbf{v}(t)$ (where $\lambda(t) \in \C \setminus\{0\}$), then
$\mathbf{w}'(t_0)=\lambda '(t_0) \mathbf{v}(t_0)+\lambda (t_0) \mathbf{v}'(t_0).$ Hence,
$\mathbf{w}(t_0)\boxtimes \mathbf{w}'(t_0)=
\overline{\lambda(t_0) ^2} \, \,\mathbf{v}(t_0)\boxtimes \mathbf{v}'(t_0).$
Likewise if $\mathbf{w}(t)=\mathbf{v}(f(t))$ for some reparametrization $f:[0,1] \to [0,1]$.

\begin{proposition}\label{invariancebypu21}
If $\gamma : [0,1] \to \partial \Hip ^2 _{\C}$ is a $C^1$ regular knot and $g \in \textrm{PU}(2,1)$ then
$$\mathcal{T}_{g\circ \gamma}(t)=g(\mathcal{T}_{\gamma}(t)).$$
\end{proposition}
\begin{proof}
If $\mathbf{v}:[0,1] \to \C^{2,1}$ is a $C^1$ lift for $\gamma$ and 
$g=[\mathbf{g}] \in \textrm{PU}(2,1)$ then
$\mathbf{g} \mathbf{v}: [0,1] \to \C ^{2,1}$ is a $C^1$ lift for $g\circ \gamma$ and
$$g\circ \gamma (t) \boxtimes (g \circ \gamma)'(t)=[\mathbf{g v}(t) \boxtimes (\mathbf{g v})'(t) ]=
[\mathbf{g}(\mathbf{v}(t)\boxtimes \mathbf{v}'(t))]=g(\gamma(t) \boxtimes \gamma '(t)).$$
Hence, $\mathcal{T}_{g \circ \gamma}(t)$ is the polar chain to the point $g(\gamma(t)\boxtimes \gamma'(t))$.
Since $g(\mathcal{T}_{\gamma}(t))$ is the polar chain to the same point, %$g(\gamma(t)\boxtimes \gamma'(t))$, 
we conclude that $\mathcal{T}_{g\circ \gamma}(t)=g(\mathcal{T}_{\gamma}(t)).$
\end{proof}

%\begin{proposition}
%If $\gamma : [0,1] \to \partial \Hip ^2 _{\C}$ is a $C^1$ regular knot then
%the image of $\gamma$ is a chain $\mathcal{C}$ if and only if 
%$\mathcal{T}_{\gamma}(t)=\mathcal{C}$ for every $t \in [0,1]$.
%\end{proposition}
%\begin{proof}
%First, we assume that the image of $\gamma$ is a chain $\mathcal{C}$. 
%By Proposition \ref{invariancebypu21} and the fact that 
%$\textrm{PU}(2,1)$ acts transitively
%on chains, we can assume that $\mathcal{C}$ is the polar chain to the point
%$[0:1:0] \in \Proy ^2$. Moreover, by Remark \ref{taniswelldef}  we can assume that 
%$\mathbf{v}:[0,1] \to \C^{2,1}$,
%$\mathbf{v}(t)=(e^{2\pi i t},0,1)$, 
%is a lift for $\gamma$. Hence,
%$\gamma(t)\boxtimes \gamma '(t)=[0:2 \pi ie^{-2 \pi i t}:0]=[0:1:0]$.
%Therefore $\mathcal{T} _{\gamma}(t)=\mathcal{C}$ for every $t \in [0,1]$.
%
%Conversely, if $\mathcal{T}_{\gamma}(t)$ is the same chain $\mathcal{C}$ for every $t \in [0,1]$, then 
%$\gamma(t)\boxtimes \gamma '(t)$ is a constant point for every $t \in [0,1]$ and it implies
%that $\gamma(t)$ lies in the polar chain to this point for every $t\in [0,1]$.
%\end{proof}

\begin{proposition}\label{prop:legendrianknot}
If $\gamma : [0,1] \to \partial \Hip ^2 _{\C}$ is a $C^1$ regular knot then the following
are equivalent
\begin{enumerate}
\item The curve $\gamma$ is Legendrian at $t$.
\item The point $\gamma(t)\boxtimes \gamma ' (t)$ lies in $\partial \Hip ^2 _{\C}$.
\item The tangent chain $\mathcal{T}_{\gamma}(t)$ is degenerate. 
\end{enumerate} 
\end{proposition}
\begin{proof}
First, 
if the curve $\mathbf{v}:[0,1] \to \C^{2,1}$ is a $C^1$ lift for $\gamma$, then 
the equation 
$$\langle \mathbf{v}(t)\boxtimes  \mathbf{v}'(t), \mathbf{v}(t)\boxtimes \mathbf{v}'(t)\rangle =
\langle \mathbf{v}'(t), \mathbf{v}(t)\rangle \langle \mathbf{v}(t), \mathbf{v}'(t)\rangle
- \langle \mathbf{v}'(t), \mathbf{v}'(t)\rangle \langle \mathbf{v}(t), \mathbf{v}(t)\rangle$$
proves the equivalence of the two equations 
$$\langle \mathbf{v}(t), \mathbf{v}'(t) \rangle=0 \quad 
\textrm{ and } \quad 
\langle \mathbf{v}(t)\boxtimes  \mathbf{v}'(t), \mathbf{v}(t)\boxtimes \mathbf{v}'(t)\rangle =0.$$
The first equation is equivalent to i) and the second one is equivalent to ii).

Finally, the equivalence of ii) and iii) follows by the definitions of $\mathcal{T}_{\gamma}(t)$ and degenerate chain.
\end{proof}

\section{Proof of the Theorem}\label{sec:proof}

\begin{lemma}\label{itislegendrian}
Let $G \subset \textrm{PU}(2,1)$ be a discrete subgroup acting 
on the complex hyperbolic space $\Hip ^2 _{\C}$. 
If the limit set $L(G) \subset \partial \Hip ^2 _{\C}$ is the image of a $C^1$ regular knot, $\gamma:[0,1] \to \partial\Hip^2_\C$, then $\gamma$ is Legendrian or a chain.
\end{lemma}

\begin{proof}
First, we claim that if $\gamma$ is Legendrian at $t_0 \in [0,1]$ 
 then $\gamma$ is Legendrian at every $t \in [0,1]$. 
%In order to prove this claim, we assume
%that $\gamma$ is Legendrian at some point $$$\gamma(\omega _0)$, then
In fact, 
$\gamma$ is Legendrian at every $t$ such that $\gamma(t)$ lies in the 
$G$-orbit of the point $\gamma(t _0)$, because the image of $\gamma$ 
is $G$-invariant and $G\subset \textrm{PU}(2,1)$ acts by contact morphisms.
For an arbitrary $t \in [0,1]$, there is a sequence of points
$t_n \to t$ as $n \to \infty$ such that $\gamma$ is Legendrian at $t_n$ for every $n\in \mathbb{N}$,
because the $G$-orbit of $\gamma(t_0)$ is dense in the image of $\gamma$. 
Hence $\gamma(t _n) \boxtimes \gamma'(t _n) \in \partial \Hip ^2 _{\C}$ 
%is a nule vector 
for every $n \in \mathbb{N}$ and it implies that 
$\gamma(t) \boxtimes \gamma '(t) \in \partial \Hip^2 _{\C}$. By Proposition~\ref{prop:legendrianknot} 
$\gamma$ is Legendrian at $t$ and we have proved our claim. 

If the image of the knot is not a chain, 
we proceed by contradiction and we assume that $\gamma$ is not Legendrian, then it is not Legendrian at
some $s_0 \in [0,1]$. It follows that 
%$\gamma(s_0) \boxtimes \gamma ' (s_1) \in\Proy ^2 \setminus \overline{\Hip ^2 _{\C}}$
%and it implies that 
$\mathcal{T}_{\gamma}(s_0)$ 
is not degenerate by Proposition~\ref{prop:legendrianknot}. Since the image of $\gamma$ is not a chain, there exists an arc of $\gamma$ disjoint from 
 $\mathcal{T}_{\gamma}(s_0)$. Moreover, there exists a loxodromic element $g\in G$ such that
its fixed points $p,q$ lie in this arc. 
By Proposition \ref{invariancebypu21}, we have that
$g^n(\mathcal{T}_{\gamma}(s_0))=\mathcal{T}_{g^n \circ \gamma}(s_0)$. 
%is the tangent chain to $\gamma$ at $g^n(\gamma(\omega))$ for every $n \in 
%\mathbb{N}$. 
If $p=\gamma(s_p)$ is the attracting fixed point for $g$, then there exists a sequence
 $(s_n) \subset [0,1]$ such that 
$\mathcal{T}_{g^n \circ \gamma}(s_0)=\mathcal{T}_{\gamma}(s_n)$
for every $n \in \mathbb{N}$, and $\lim\limits _{n\to \infty} s_n = s_p$.
It follows %from a property ofthe Proposition \ref{northsouth} 
that the diameter of the chain $g^n(\mathcal{T}_{\gamma}(s_0))=\mathcal{T}_{\gamma}(s_n)$
goes to zero as $n \to \infty$. Hence $\mathcal{T}_{\gamma}(s_p)$ is degenerate, or equivalently by Proposition~\ref{prop:legendrianknot},  
$\gamma$ is Legendrian at $s_p$.
Therefore, $\gamma$ is Legendrian at any $s \in [0,1]$, which is a contradiction.
\end{proof}

\begin{lemma}\label{lem:embedded-plane-curve}
Let $I$ be an open interval such that 
$\gamma: I \to \C$ is an embedded plane curve of class $C^2$ 
invariant under the homothety $g(z) = \lambda\,z$, for 
some $\lambda \in \C$, $|\lambda| < 1$, then 
$\gamma$ is a straight line segment.
\end{lemma}
\begin{proof}
Let $k(z)$ be the the curvature of $\gamma$ at a point $z = \gamma(t)$, 
$k$ is a geometric invariant of $\gamma(I)$, independent of the 
parametrization chosen. Let $z_0 = \gamma(a)$ be any point of the curve, then 
the curvature at the $n$-th iterate $z_{n} = g^{n}\circ \gamma(a)$ 
is $k(z_{n}) = |\lambda|^{-n}\,k(z_{0})$, but $z_n \to 0$ as $n \to \infty$ 
whereas $k(z_{n}) \to \infty$ unless $k(z_0) = 0$, since $z_0$ was arbitrary, this 
means that $\gamma$ has constant curvature $k \equiv 0$, hence it is a 
straight line segment.
\end{proof}
\begin{lemma}\label{lem:r-circle}
If $\gamma : [0,1] \to \partial\Hip^2_\C$ is a $C^2$ regular Legendrian knot 
and there is a loxodromic element $g \in \pu (2,1)$ such that ${\gamma}$ is 
$g$-invariant and contains the attracting and repelling fixed points of $g$, then 
$\gamma$ is an $\reals$-circle.
\end{lemma}
\begin{proof}
Let $p,q$ be the attracting and repelling fixed points of $g$ respectively. 
We can choose Heisenberg coordinates for $\partial\Hip_\C^2 \setminus{\{q\}}$ 
such that $p$ is the origin in $\Heis$ and $q$ is the point at infinity. Under 
these coordinates,  $g$ has the representation 
\begin{equation*}
    g(\zeta,v) = (\lambda^{1/2}\zeta, |\lambda|\,v),
\end{equation*}
where $|\lambda| < 1$ since the attracting fixed point is at the origin. Identifying $\gamma$ with its image under Heisenberg coordinates, we can suppose it is a curve 
$\reals \to \Heis$ joining the origin with the point at infinity, such that $\gamma(0) = \vb 0$. Let $\psi: \reals \to \C$ be the vertical projection of $\gamma$, since 
$\gamma$ is Legendrian, $\psi'(0) \neq 0$, hence by the inverse function theorem,
there is an open interval $I$ centered at $0$ such that the restriction $\psi|_{I}$ is embedded 
in the complex plane at $v=0$. 
By the Lemma~\ref{lem:embedded-plane-curve}, $\psi|_{I}$ is a 
straight line segment, since $\gamma$ is tangent to the contact structure, this means that 
$\gamma(I)$ is this segment. Also, $\gamma$ is $g$-invariant and connected, hence  
$\gamma(\reals) = \cup_{n \in \mathbb{N}} g^{-n}\circ\gamma\left({I}\right)$ is a horizontal 
line and therefore it is an $\reals$-circle.
\end{proof}

\begin{proof}
    [Proof of the Theorem~\ref{main}]
    By the Lemma~\ref{itislegendrian}, if $L(G)$ is not a chain, then it is a 
    Legendrian curve in $\partial\Hip^2_\C$. We know that in any complex Kleinian 
    group $G \subset \textrm{PU}(2,1)$, there is a loxodromic element $g$, since 
    $L(G)$ is $g$-invariant and contains both the attracting and repelling fixed points of 
    $g$, by the Lemma~\ref{lem:r-circle} it is an $\reals$-circle.
\end{proof}
	\textbf{Acknowledgements.} The research of W. Barrera, R. Garc\'ia and J. P. Navarrete
  has been supported by the CONACYT, Proyecto Ciencia de Frontera
  2019--21100 via the Faculty of Mathematics, UADY.

	\bibliography{main.bib}
	\bibliographystyle{plain}
\end{document}